\documentclass[11pt,a4paper,reqno]{amsart}
\usepackage{amsfonts,caption,amsmath,amssymb,enumerate,hyperref,latexsym,xcolor,amsthm,multicol,cases,tikz}
\usetikzlibrary{calc}

\newcommand{\A}{A}
\newcommand{\AGL}{\mathrm{AGL}}
\newcommand{\bbF}{\mathbb{F}}
\newcommand{\Fix}{\mathrm{Fix}}
\newcommand{\fpr}{\mathrm{fpr}}
\newcommand{\GL}{\mathrm{GL}}
\newcommand{\PGaL}{\mathrm{P\Gamma L}}
\newcommand{\PGaU}{\mathrm{P\Gamma U}}
\newcommand{\PGL}{\mathrm{PGL}}
\newcommand{\PSL}{\mathrm{PSL}}
\newcommand{\rmO}{\mathrm{O}}
\newcommand{\SL}{\mathrm{SL}}
\newcommand{\Soc}{\mathrm{Soc}}
\newcommand{\Sp}{\mathrm{Sp}}
\newcommand{\Sym}{\mathrm{Sym}}
\newcommand{\Sy}{S}
\newcommand{\Sh}{\mathrm{Sh}}
\newcommand{\Z}{\mathbf{Z}}

\newtheorem{theorem}{Theorem}[section]

\newtheorem{lemma}[theorem]{Lemma}

\theoremstyle{definition}

\newtheorem{question}[theorem]{Question}
\newtheorem{conjecture}[theorem]{Conjecture}

\newtheorem*{remark}{Remark}

\begin{document}
\title[A complete classification of shuffle groups]{A complete classification of shuffle groups}

\author[Xia]{Binzhou Xia}
\address{School of Mathematics and Statistics\\The University of Melbourne\\Parkville, VIC 3010\\Australia}
\email{binzhoux@unimelb.edu.au}

\author[Zhang]{Junyang Zhang}
\address{School of Mathematical Sciences\\Chongqing Normal University\\Chongqing, 401331\\P.R.~China}
\email{jyzhang@cqnu.edu.cn}

\author[Zhang]{Zhishuo Zhang}
\thanks{Corresponding author: Zhishuo Zhang}
\address{School of Mathematics and Statistics\\The University of Melbourne\\Parkville, VIC 3010\\Australia}
\email{zhishuoz@student.unimelb.edu.au}

\author[Zhu]{Wenying Zhu}
\address{School of Mathematical Sciences, and Hebei Center for Applied Mathematics\\Hebei Normal University\\
Shijiazhuang, 050024\\P.R.~China}
\email{zfwenying@mail.bnu.edu.cn}

% \received{04 May 2020}
% %\revised{10 May 2020}
% \accepted{12 May 2020}

% \keywords{perfect shuffle, shuffle group, $2$-transitive permutation group, fixed point ratio.}

% \keywords[MSC2020]{20B35}

\begin{abstract}
For positive integers $k$ and $n$, the shuffle group $G_{k,kn}$ is generated by the $k!$ permutations of a deck of $kn$ cards performed by cutting the deck into $k$ piles with $n$ cards in each pile, and then perfectly interleaving these cards following a certain permutation of the $k$ piles. For $k=2$, the shuffle group $G_{2,2n}$ was determined by Diaconis, Graham and Kantor in 1983. The Shuffle Group Conjecture states that, for general $k$, the shuffle group $G_{k,kn}$ contains $\mathrm{A}_{kn}$ whenever $k\notin\{2,4\}$ and $n$ is not a power of $k$. In particular, the conjecture in the case $k=3$ was posed by Medvedoff and Morrison in 1987. The only values of $k$ for which the Shuffle Group Conjecture has been confirmed so far are powers of $2$, due to recent work of Amarra, Morgan and Praeger based on Classification of Finite Simple Groups. In this paper, we confirm the Shuffle Group Conjecture for all cases using results on $2$-transitive groups and elements of large fixed point ratio in primitive groups.


\end{abstract}
\maketitle
% \end{Frontmatter}

% %\localtableofcontents

% \vspace*{14pt}

\section{Introduction}\label{sec1}

%Perfect shuffles of cards are often used in magician's card tricks, which adds more mysteries to the tricks, and these mysteries can be traced using mathematics.
For a deck of $2n$ cards, the usual way to perfectly shuffle the deck is to first cut the deck in half (see Figure~1) and then perfectly interleave the two halves.
There are two kinds of such shuffles according to whether the original top card remains on top or not (see Figures~2 and~3).
Note that these two shuffles are permutations of the $2n$ cards. To exactly know what permutations of the cards can be achieved by performing a sequence of these two shuffles, one needs to determine the permutation group generated by these two shuffles.
In 1983, Diaconis, Graham and Kantor~\cite{DGK1983} completely determined this group for all $n$ (see Theorem~\ref{thm:2k}).
Moreover, at the end of~\cite{DGK1983}, they suggested a more general problem: For an integer $k\geq2$, if a deck of $kn$ cards are divided into $k$ piles with $n$ cards in each pile, then there are $k!$ possible orders of picking up the piles to perfectly interleave. Therefore, there are $k!$ such ways to perfectly shuffle the $kn$ cards, and one may consider the group generated by these $k!$ permutations.

\vspace{15pt}
{\small
\begin{center}
\begin{tikzpicture}
\draw  (-1,-3.4) node {\textsc{Figure} 1:~Cut the deck};
\draw  (5,-3.4) node {\textsc{Figure} 2:~Out-shuffle};
\draw  (9,-3.4) node {\textsc{Figure} 3:~In-shuffle};
\tikzstyle{every node}=[draw,circle,fill=white,minimum size=0pt,inner sep=0pt]
\draw (-2.8,0) node (0) [label=left:$0$] {}
      (-2.8,-0.6) node (1) [label=left:$1$] {}
      (-2.65,-1.2)  node[fill=black,minimum size=2pt,inner sep=0pt] {}
      (-2.65,-1.5)  node[fill=black,minimum size=2pt,inner sep=0pt] {}
      (-2.65,-1.8)  node[fill=black,minimum size=2pt,inner sep=0pt] {}
      (-2.8,-2.4)  node (2n-1) [label=left:$2n-1$] {};
\draw (0,0) node (3) [label=left:$0$] {}
      (0,-0.6) node (4) [label=left:$1$] {}
      (0.25,-1.2)  node[fill=black,minimum size=2pt,inner sep=0pt] {}
      (0.25,-1.5)  node[fill=black,minimum size=2pt,inner sep=0pt] {}
      (0.25,-1.8)  node[fill=black,minimum size=2pt,inner sep=0pt] {}
      (0,-2.4)  node (5) [label=left:$n-1$] {};
\draw (1.25,0) node (0.22) [label=right:$n$] {}
      (1.25,-0.6) node (1.22) [label=right:$n+1$] {}
      (1,-1.2)  node[fill=black,minimum size=2pt,inner sep=0pt] {}
      (1,-1.5)  node[fill=black,minimum size=2pt,inner sep=0pt] {}
      (1,-1.8)  node[fill=black,minimum size=2pt,inner sep=0pt] {}
      (1.25,-2.4)  node (2n-1.2) [label=right:$2n-1$] {};
\draw (4.6,0) node (0.31) [label=left:$0$] {}
      (5.4,-0.3) node (0.32) [label=right:$n$] {}
      (4.6,-0.6) node (1.31) [label=left:$1$] {}
      (5.4,-0.9) node (1.32) [label=right:$n+1$] {}
      (5,-1.3)  node[fill=black,minimum size=2pt,inner sep=0pt] {}
      (5,-1.6)  node[fill=black,minimum size=2pt,inner sep=0pt] {}
      (5,-1.9)  node[fill=black,minimum size=2pt,inner sep=0pt] {}
      (4.6,-2.3)  node (2.31) [label=left:$n-1$] {}
      (5.4,-2.6)  node (2.32) [label=right:$2n-1$] {};
\draw (8.6,0) node (15) [label=left:$n$] {}
      (9.4,-0.3) node (16) [label=right:$0$] {}
      (8.6,-0.6) node (17) [label=left:$n+1$] {}
      (9.4,-0.9) node (18) [label=right:$1$] {}
      (9,-1.3)  node[fill=black,minimum size=2pt,inner sep=0pt] {}
      (9,-1.6)  node[fill=black,minimum size=2pt,inner sep=0pt] {}
      (9,-1.9)  node[fill=black,minimum size=2pt,inner sep=0pt] {}
      (8.6,-2.3)  node (19) [label=left:$2n-1$] {}
      (9.4,-2.6)  node (20) [label=right:$n-1$] {};
\draw (-1.7,-1.2)  node (8) [fill=black,minimum size=0pt,inner sep=0pt]{}
      (-1.1,-1.2)  node (9)[fill=black,minimum size=0pt,inner sep=0pt]{};
\draw [fill=black,very thick,->](8) -- (9);
\draw [thick,-] (-2.8,0) -- (-2.3,0)
                (-2.8,-0.6) -- (-2.3,-0.6)
                (-2.8,-2.4) -- (-2.3,-2.4)
                (0,0)     --  (0.5,0)
                (0,-0.6)  --  (0.5,-0.6)
                (0,-2.4)  --  (0.5,-2.4)
                (0.75,0)  --  (1.25,0)
                (0.75,-0.6)  --  (1.25,-0.6)
                (0.75,-2.4)  --  (1.25,-2.4)
                (4.6,0)  --  (5.1,0)
                (4.9,-0.3)  --  (5.4,-0.3)
                (4.6,-0.6)  --  (5.1,-0.6)
                (4.9,-0.9)  --  (5.4,-0.9)
                (4.6,-2.3)  --  (5.1,-2.3)
                (4.9,-2.6)  --  (5.4,-2.6)
                (8.6,0)  --  (9.1,0)
                (8.9,-0.3)  --  (9.4,-0.3)
                (8.6,-0.6)  --  (9.1,-0.6)
                (8.9,-0.9)  --  (9.4,-0.9)
                (8.6,-2.3)  --  (9.1,-2.3)
                (8.9,-2.6)  --  (9.4,-2.6);
\end{tikzpicture}
\end{center}
}
\vspace{12pt}

Throughout this paper, for a positive integer $m$, we set
\[
[m]=\{0,1,\ldots,m-1\}.
\]
For a deck of $kn$ cards, the card in position $i+jn$, where $i\in[n]$ and $j\in[k]$, refers to the $(i+jn)$-th card from top to bottom with the top one being the $0$-th card.
We may also think of them in $k$ piles such that the $j$-th pile consists of the cards in the positions $jn,1+jn,\dots,n-1+jn$, where $j\in\{0,1,\dots,k-1\}$.
The \emph{standard shuffle} of the $kn$ cards, denoted by $\sigma$, is performed by picking up the top card from each of the piles $0,\dots,k-1$ in order and repeating until all cards have been picked up, that is, $\sigma$ is the permutation of $[kn]$ defined by
\[
(i+jn)^{\sigma}=ik+j\ \text{ for all }\ i\in[n]\,\text{ and }\,j\in[k].
\]
Let $\tau\in\Sym([k])$ be a permutation of the $k$ plies. Then $\tau$ induces a permutation $\rho_\tau$ of the $kn$ cards by keeping the order of the cards within each plie, so that
\begin{equation}\label{Eqn:rho}
(i+jn)^{\rho_\tau}=i+j^\tau n\ \text{ for all }\ i\in[n]\,\text{ and }\,j\in[k].
\end{equation}
The card shuffle $\rho_\tau\sigma$ is to first perform $\rho_\tau$ and then $\sigma$, that is,
\[
(i+jn)^{\rho_\tau\sigma}=(i+j^\tau n)^{\sigma}=ik+j^\tau\ \text{ for all }\ i\in[n]\,\text{ and }\,j\in[k].
\]
The subgroup of $\Sym([kn])$ generated by $\rho_\tau\sigma$ for all $\tau\in\Sym([k])$ is called the \emph{shuffle group} on $kn$ cards and denoted by $G_{k,kn}$, that is,
\[
G_{k,kn}=\langle\rho_\tau\sigma\mid\tau\in\Sym([k])\rangle=\langle\sigma,\rho_\tau\mid\tau\in\Sym([k])\rangle.
\]
In this terminology, what is suggested at the end of~\cite{DGK1983} is to determine the shuffle group $G_{k,kn}$.

For a positive integer $m$, let $C_m$, $A_m$ and $S_m$ be the cyclic group of order $m$, alternating group on $m$ points and symmetric group on $m$ points, respectively.
In 1983, Diaconis, Graham and Kantor~\cite{DGK1983} completely determined $G_{2,2n}$ as follows.

\begin{theorem}\label{thm:2k}
\emph{(Diaconis-Graham-Kantor)} For $G=G_{2,2n}$, the following hold:
  \begin{enumerate}[\rm(a)]
    \item If $n\equiv 0\pmod{4}$, $n>12$ and $n$ is not a power of $2$, then $G=C^{n-1}_2\rtimes A_n$.
    \item If $n\equiv 1\pmod{4}$, then $G=C^n_2\rtimes A_n$.
    \item If $n\equiv 2\pmod{4}$ and $n>6$, then $G$ is the imprimitive wreath product $C_2\wr S_n$.
    \item If $n\equiv 3\pmod{4}$, then $G=C^{n-1}_2\rtimes S_n$ is the Weyl group of the root system $D_n$.
    \item If $n=2^m$ for some positive integer $m$, then $G$ is the primitive wreath product $C_2\wr C_{m+1}$.
    \item If $n=6$, then $G=C_2^6\rtimes \PGL(2,5)$.
    \item If $n=12$, then $G=C_2^{11}\rtimes M_{12}$, where $M_{12}$ is the Mathieu group on $12$ points.
  \end{enumerate}
\end{theorem}
%
%\begin{remark}
%The group $G$ in part~(d) of Theorem~\ref{thm:2k} is isomorphic to the Weyl group of the root system $D_n$.
%\end{remark}

In 1987, Medvedoff and Morrison~\cite{MM1987} initiated a systematic study of the shuffle group $G_{k,kn}$ for general $k$. They showed (see~\cite[Theorem~2]{MM1987}) that, if $n$ is a power of $k$, then $G_{k,kn}$ is the primitive wreath product of $\Sy_k$ by the cyclic group of order $\log_k(kn)$, that is,
\begin{equation}\label{eq:wr}
      G_{k,k^m}=S_k\wr C_m.
\end{equation}
 Moreover, based on computation results, Medvedoff and Morrison conjectured in~\cite{MM1987} that $G_{3,3n}$ contains $\A_{3n}$ if $n$ is not a power of $3$, which is essentially a conjectural classification of $G_{3,3n}$ (see the Remark after Conjecture~\ref{conk}). Similarly, they made a conjecture for $k=4$ in the same paper, which states that $G_{4,4n}$ contains $\A_{4n}$ if $n$ is not a power of $2$, and $G_{4,4n}$ is the full affine group of degree $4n$ if $n$ is an odd power of $2$.
The latter part of this conjecture, that is,
\begin{equation}\label{eq:aff}
      G_{4,2^{2\ell+1}}=\AGL(2\ell+1,2),
\end{equation} was confirmed in 2005 by Cohen, Harmse, Morrison and Wright~\cite[Theorem~2.6]{CHMW2005}. This leads them to the following conjecture.

\begin{conjecture}\label{conk}
(Shuffle Group Conjecture) For $k\geq3$, if $n$ is not a power of $k$ and $(k,n)\neq(4,2^f)$ for any positive integer $f$, then $G_{k,kn}$ contains $\A_{kn}$.
\end{conjecture}

\begin{remark}
From the definitions of $\sigma$ and $\rho_\tau$, it is not hard to see that $G_{k,kn}$ is a subgroup of $\A_{kn}$ if and only if either $n\equiv2\pmod4$ and $k\equiv0$ or $1\pmod4$, or $n\equiv0\pmod4$ (see, for example, ~\cite[Theorem~1]{MM1987}). This implies that $G_{k,kn}$ is precisely determined if we know that $G_{k,kn}$ contains $\A_{kn}$. Therefore, the above Shuffle Group Conjecture is in fact a conjectural classification of the shuffle groups $G_{k,kn}$ for all $k\geq3$ and $n\geq1$.
\end{remark}

Among other results, Amarra, Morgan and Praeger~\cite{AMP2021} recently confirmed Conjecture~\ref{conk} for the following three cases:
\begin{enumerate}[\rm(i)]
\item $k>n$;
%\item $(k,n)=(\ell^s,\ell^t)$ for some integers $s$, $t$ and $\ell\geq2$;
\item $k$ and $n$ are powers of the same integer $\ell\geq2$;
\item $k$ is a power of $2$, and $n$ is not a power of $2$.\label{state}
\end{enumerate}
Note that~(ii) and~(iii) together imply the validity of Conjecture~\ref{conk} whenever $k$ is a power of $2$.
%In particular, the conjecture holds for $k=4$.
We also remark that the Classification of Finite Simple Groups (CFSG) comes into play in the study of shuffle groups in~\cite{AMP2021}. In fact, CFSG was already applied in an unpublished result of William Kantor (see~\cite{K} and~\cite[Page~13]{MM1987}) to prove that $G_{k,kn}\geq A_{kn}$ if $k\geq 4$ and $k$ does not divide $n$.

In this paper, we prove Conjecture~\ref{conk} for all $k$ and $n$ (see Theorems~\ref{thmreduct} and~\ref{thm2tran}). Our approach is to reduce the proof of the conjecture to that of the $2$-transitivity of $G_{k,kn}$ by considering the fixed point ratio of certain element therein. This approach makes use of some deep classification results (depending on CFSG) in group theory, for example, the classification of $2$-transitive groups and primitive groups with elements of large fixed point ratio~\cite{Burness2018,GM1998,LS1991}. Our reduction theorem is as follows.

\begin{theorem}\label{thmreduct}
      If $G_{k,kn}$ is $2$-transitive with $k\geq3$, then either $k=4$ and $n$ is an odd power of $2$, or $G_{k,kn}$ contains $\A_{kn}$.
\end{theorem}

By Theorem~\ref{thmreduct}, we can determine $G_{k,kn}$ for $k\geq 3$ if it is shown to be $2$-transitive. This is the case when $n$ is not a power of $k$, as the following theorem states.

\begin{theorem}\label{thm2tran}
  The shuffle group $G_{k,kn}$ is $2$-transitive if $k\geq 3$ and $n$ is not a power of $k$.
\end{theorem}

The combination of Theorems~\ref{thmreduct} and \ref{thm2tran} completely solves Conjecture~\ref{conk} affirmatively.
Now that Conjecture~\ref{conk} is confirmed, it together with \cite[Theorem 2]{MM1987} and \cite[Theorem~2.6]{CHMW2005} leads to (see the remark after Conjecture~\ref{conk}) the following complete classification of shuffle groups.

% \begin{theorem}\label{classification}
%       If $k\geq 3$, then the following hold:
%       \begin{enumerate}[\rm(a)]
%             \item If $kn=k^m$, then $G_{k,kn}$ is the primitive wreath product $S_k\wr C_m$.
%             \item If $k=4$ and $kn=2^m$ with $m$ odd,  then $G_{k,kn}$ is the affine group $\AGL(m,2)$.
%             \item If either $n\equiv2\pmod4$ and $k\equiv0$ or $1\pmod4$ with $(k,n)\neq(4,2)$, or $n\equiv0\pmod4$ and $n$ is not a power of $k$, then $G_{k,kn}=A_{kn}$.
%             \item In all other cases, $G_{k,kn}=S_{kn}$.
%       \end{enumerate}
% \end{theorem}

% \begin{theorem}\label{classification}
%   If $k\geq 3$, then the following hold:
%   \begin{enumerate}[\rm(a)]
%         \item If $kn=k^m$, then $G_{k,kn}$ is the primitive wreath product $S_k\wr C_m$.
%         \item If $k=4$ and $kn=2^m$ with $m$ odd,  then $G_{k,kn}$ is the affine group $\AGL(m,2)$.
%         \item If $n$ is not a power of $k$ and either $n\equiv 2\pmod4$ and $k(k-1)/2$ is odd or $n$ is odd, then $G_{k,kn}=S_{kn}$.
%         \item In all other cases, $G_{k,kn}=A_{kn}$.
%   \end{enumerate}
% \end{theorem}

\begin{theorem}\label{classification}
  If $k\geq 3$, then the following hold:
  \begin{enumerate}[\rm(a)]
        \item If $kn=k^m$, then $G_{k,kn}$ is the primitive wreath product $S_k\wr C_m$.
        \item If $k=4$ and $kn=2^m$ with $m$ odd,  then $G_{k,kn}$ is the affine group $\AGL(m,2)$.
        \item If $n$ is not a power of $k$ and either $n$ is odd or both $n/2$ and $k(k-1)/2$ are odd integers, then $G_{k,kn}=S_{kn}$.
        \item In all other cases, $G_{k,kn}=A_{kn}$.
  \end{enumerate}
\end{theorem}

% Our proof of Theorem~\ref{thmreduct} mainly relies on the existence of a permutation of $\Sym([kn])$ with fixed point ratio no less than $1/2$. This indicates that we may obtain a similar result to Theorem~\ref{thmreduct} for $\Sh(P,n)$ where $P< \Sym([k])$ as long as $\Sh(P,n)$ contains a permutation with large fixed point ratio.

The remainder of this paper is structured as follows. In the next section, we will give definitions and some technical lemmas that will be used in Section~\ref{sec3}. After this preparation, Theorem~\ref{thmreduct} will be proved in Section~\ref{sec3}, and the proof of Theorems~\ref{thm2tran} will be given in Section~\ref{sec4}. Finally, in Section~\ref{sec5}, we conclude the paper with some open problems on the so-called generalised shuffle groups introduced by Amarra, Morgan and Praeger~\cite{AMP2021}.

\section{Preliminaries}\label{sec2}

For a finite group $G$, let $\Z(G)$ denote the centre of $G$, let $\mathbf{O}_p(G)$ denote the largest normal $p$-subgroup of $G$ for a prime $p$, and let $\mathbf{C}_G(g)$ denote the centraliser of an element $g$ in $G$. The \emph{socle} of $G$ is the product of the minimal normal subgroups of $G$, denoted by $\Soc(G)$. The \emph{fixed point ratio} of a permutation $g$ on a finite set $\Omega$, denoted by $\fpr(g)$, is defined by
\[
\fpr(g)=\frac{|\Fix(g)|}{|\Omega|},
\]
where $\Fix(g)=\{\alpha\in\Omega\mid\alpha^g=\alpha\}$.
%
%{\color{blue}The following lemma is not hard to see, and we prove it here for self-sufficiency.}

\begin{lemma}\label{lempgl}
Let $g$ be an element of $\PGL(d,3)$ acting on the set of $1$-dimensional subspaces of the vector space $\bbF_3^d$. Then
\[
\fpr(g)=\frac{3^s+3^t-2}{3^d-1}
\]
for some nonnegative integers $s$ and $t$.
\end{lemma}

\begin{proof}
%Consider the action of $\PGL(d,3)=\GL(d,3)/\Z(\GL(d,3))$ on the set of $1$-dimension subspace of $\mathbb{F}^d_3$.
Let $\hat{g}\in\GL(d,3)$ such that $g=\hat{g}\Z(\GL(d,3))\in\PGL(d,3)$. Note that a $1$-dimensional subspace $\langle v\rangle$ of $\bbF^d_3$ satisfies $\langle v\rangle^g=\langle v\rangle$ if and only if $v^{\hat{g}}$ is $v$ or $-v$. Therefore,
\[
\Fix(g)=\{\langle v\rangle\mid v\in\bbF^d_3\setminus\{0\},\,v^{\hat{g}}=v\}\cup\{\langle v\rangle\mid v\in\bbF^d_3\setminus\{0\},\,v^{\hat{g}}=-v\},
\]
and hence
\[
\fpr(g)=\frac{|\Fix(g)|}{|\{\langle v\rangle\mid v\in\mathbb{F}^d_3\setminus\{0\}\}|}=\frac{\frac{3^s-1}{2}+\frac{3^t-1}{2}}{\frac{3^d-1}{2}}=\frac{3^s+3^t-2}{3^d-1},
\]
where $s$ and $t$ are the dimensions of the $1$-eigenspace and ($-1$)-eigenspace of $\hat{g}$, respectively.
\end{proof}

%To prepare for the proof of Proposition~\ref{proeq3}, we
Let $V=\bbF_q^d$ be a $d$-dimensional vector space over $\mathbb{F}_q$, where $d\geq3$ and $q$ is even, and we fix an ordered basis of $V$ and associate each element of $\SL(V)$ with its matrix under this basis. For an involution $g\in\SL(V)$, denote by $r(g)$ the number of Jordan blocks of size $2$ in the Jordan canonical form of $g$. Note that two involutions $A$ and $B$ in $\SL(V)$ are conjugate in $\SL(V)$ if and only if $r(A)=r(B)$. For an integer $\ell$ with $1\leq\ell\leq d/2$, denote
\[
A_\ell=\begin{pmatrix}I_\ell & & \\ & I_{d-2\ell} & \\ I_\ell & & I_\ell\end{pmatrix},
\]
where $I_j$ is the $j\times j$ identity matrix. It is clear that $A_\ell$ is an involution in $\SL(V)$ with $r(A_\ell)=\ell$. We call $A_\ell$ the \emph{Suzuki form} of the conjugacy class of $A_\ell$ in $\SL(V)$.

For $\varepsilon\in\{+,-\}$, let $\rmO^{\varepsilon}(2m,q)$ be the general orthogonal group of $\varepsilon$ type on the space $\bbF_q^{2m}$, where $m$ is a positive integer and $q$ is a prime power. For convenience, we set the notation $\Sp(0,q)$ and $\rmO^{\varepsilon}(0,q)$ to be the trivial group. The following lemma is a consequence of~\cite[Sections~7 and~8]{AS1976}.

\begin{lemma}\label{lemspo}
For each involution $g\in\rmO^{\varepsilon}(2m,2)<\Sp(2m,2)$, we have
\[
\frac{|\mathbf{C}_{\Sp(2m,2)}(g)|}{|\mathbf{C}_{\rmO^{\varepsilon}(2m,2)}(g)|}
=\frac{|\Sp(2m-2r,2)|\cdot|\mathbf{O}_2\left(\mathbf{C}_{\Sp(2m,2)}(g)\right)|}
{|\rmO^{\varepsilon}(2m-2r,2)|\cdot|\mathbf{O}_2\left(\mathbf{C}_{\rmO^{\varepsilon}(2m,2)}(g)\right)|}
%\ \text{ or }\ \frac{|\mathbf{O}_2\left(\mathbf{C}_{\Sp(2m,2)}(x)\right)|}{|\mathbf{O}_2\left(\mathbf{C}_{\rmO^{\varepsilon}(2m,2)}(x)\right)|},
\]
for some positive integer $r\leq m$.
%where $\ell$ is the number of Jordan blocks of size $2$ in the Jordan canonical form of $x$.
\end{lemma}

\begin{proof}
Write $G=\Sp(2m,2)$ and $H=\rmO^{\varepsilon}(2m,2)$.
%Let $x$ be an involution of $H$, and let $\ell$ be as in the previous paragraph.
Since $g\in G$, we see that there exists a basis of $\mathbb{F}^{2m}_2$ as in (1), (2) or (3) of~\cite[(7.6)]{AS1976} such that $g$ is in Suzuki form under this basis. For convenience, we say that $g$ has form $a_\ell$, $b_\ell$ or $c_\ell$, if the basis is chosen as in (1), (2) or (3) of~\cite[(7.6)]{AS1976}, respectively.

First assume that $g$ has form $a_\ell$ (in this case, $\ell$ is even).
It follows from~\cite[(7.9)]{AS1976} that there exists a homomorphism from $\mathbf{C}_G(g)$ onto $\Sp(\ell,2)\times\Sp(2m-2\ell,2)$ with kernel $\mathbf{O}_2\left(\mathbf{C}_{G}(g)\right)$. Therefore,
\[
\frac{|\mathbf{C}_{G}(g)|}{|\mathbf{O}_2\left(\mathbf{C}_{G}(g)\right)|}=|\Sp(\ell,2)\times\Sp(2m-2\ell,2)|.
\]
Moreover, \cite[(8.6)]{AS1976} shows that there is a homomorphism from $\mathbf{C}_{H}(g)$ to $\Sp(\ell,2)\times\rmO^{\varepsilon}(2m-2\ell,2)$ with kernel $\mathbf{O}_2\left(\mathbf{C}_{H}(g)\right)$, and so
\[
\frac{|\mathbf{C}_{H}(g)|}{|\mathbf{O}_2\left(\mathbf{C}_{H}(g)\right)|}=|\Sp(\ell,2)\times\rmO^{\varepsilon}(2m-2\ell,2)|.
\]
As a consequence,
\[
\frac{|\mathbf{C}_{G}(g)|}{|\mathbf{C}_{H}(g)|}
=\frac{|\Sp(2m-2\ell,2)|\cdot|\mathbf{O}_2\left(\mathbf{C}_{G}(g)\right)|}
{|\rmO^{\varepsilon}(2m-2\ell,2)|\cdot|\mathbf{O}_2\left(\mathbf{C}_{H}(g)\right)|}.
\]

Now assume that $g$ has form $b_\ell$ or $c_\ell$ (in this case, $\ell$ is odd or even, respectively).
Similarly, we derive from~\cite[(7.10)~and~(7.11)]{AS1976} and~\cite[(8.7)~and~(8.8)]{AS1976} that
\[
\frac{|\mathbf{C}_{G}(g)|}{|\mathbf{O}_2\left(\mathbf{C}_{G}(g)\right)|}=|\Sp(\ell-1,2)\times\Sp(2m-2\ell,2)|=
\frac{|\mathbf{C}_{H}(g)|}{|\mathbf{O}_2\left(\mathbf{C}_{H}(g)\right)|}
\]
or
\[
\frac{|\mathbf{C}_{G}(g)|}{|\mathbf{O}_2\left(\mathbf{C}_{G}(g)\right)|}=|\Sp(\ell-2,2)\times\Sp(2m-2\ell,2)|=
\frac{|\mathbf{C}_{H}(g)|}{|\mathbf{O}_2\left(\mathbf{C}_{H}(g)\right)|}.
\]
It follows that
\[
\frac{|\mathbf{C}_{G}(g)|}{|\mathbf{C}_{H}(g)|}=\frac{|\mathbf{O}_2\left(\mathbf{C}_{G}(g)\right)|}{|\mathbf{O}_2\left(\mathbf{C}_{H}(g)\right)|}
=\frac{|\Sp(0,2)|\cdot|\mathbf{O}_2\left(\mathbf{C}_{G}(g)\right)|}
{|\rmO^{\varepsilon}(0,2)|\cdot|\mathbf{O}_2\left(\mathbf{C}_{H}(g)\right)|}.
\]
This completes the proof.
\end{proof}

\section{Fixed point ratio and shuffle groups}\label{sec3}

In this section, we prove the reduction theorem (Theorem~\ref{thmreduct}), which reduces the proof of Conjecture~\ref{conk} to that of the $2$-transitivity of $G_{k,kn}$. Recall from~\eqref{Eqn:rho} that for $\tau\in\Sym([k])$, the permutation $\rho_\tau\in G_{k,kn}$ maps $i+jn$ to $i+j^\tau n$ for all $i\in[n]$ and $j\in[k]$. This leads to the following result on the fixed point ratio of $\rho_\tau$, an observation that is the basis of our argument throughout this section.

\begin{lemma}\label{lemfpr}
For each $\tau\in\Sym([k])$ we have $\fpr(\tau)=\fpr(\rho_{\tau})$. In particular, if $\tau$ is a transposition, then $\fpr(\rho_{\tau})=(k-2)/k$.
\end{lemma}

A permutation group $G$ on a set $\Omega$ is said to be \emph{primitive} if the only partitions preserved by $G$ are $\{\Omega\}$ and $\{\{\alpha\}\mid\alpha\in\Omega\}$. It is well known and easy to see that every $2$-transitive group is primitive. An \emph{affine} primitive group is a subgroup of $\AGL(d,p)$ that contains the socle of $\AGL(d,p)$, where $d$ is a positive integer and $p$ is prime.

\begin{lemma}\label{lemaffine}
Suppose that $G_{k,kn}$ is an affine primitive group with $k\geq3$. Then either $k=3$ and $n=1$, or $k=4$ and $n$ is a power of $2$.
\end{lemma}

\begin{proof}
Let $G=G_{k,kn}$, and let $V$ be a $d$-dimension vector space over $\mathbb{F}_p$ such that $G\leq\AGL(V)$, where $d$ is a positive integer and $p$ is prime. Then $kn=|V|=p^d$. By~\eqref{Eqn:rho}, there is a transposition $\tau\in\Sym([k])$ such that $\rho_\tau$ fixes the zero vector $0$ in $V$. It follows that $\rho_\tau\in G_0\leq\GL(V)$. Since $\Fix(\rho_\tau)=\{v\in V\mid v^{\rho_\tau}=v\}$ is a subspace of $V$, we have $|\Fix(\rho_\tau)|=p^f$ for some nonnegative integer $f$. Thus, as $\tau$ is a transposition, we derive from Lemma~\ref{lemfpr} that
\[
\frac{k-2}{k}=\fpr(\rho_\tau)=\frac{|\Fix(\rho_\tau)|}{|V|}=\frac{p^f}{p^d}=\frac{1}{p^{d-f}}.
\]
Since $k\geq3$, this implies that either $k=p=3$, or $k=4$ and $p=2$. For the latter, $n=|V|/k=2^{d-2}$ is a power of $2$. Now assume that $k=p=3$. Then $n=|V|/k=p^d/k=3^{d-1}$, and so~\eqref{eq:wr} gives $G=\Sy_3\wr C_{d}$. Since $G$ is affine, we conclude that $d=1$, which indicates that $n=3^{d-1}=1$.
\end{proof}

A group is said to be \emph{almost simple} if its socle is a nonabelian simple group. It follows from the well-known Burnside's Theorem \cite[\S 154, Theorem XIII]{B1911} that $2$-transitive groups are either affine or almost simple.

\begin{proof}[Proof of Theorem~$\ref{thmreduct}$]
Let $G=G_{k,kn}$ be $2$-transitive with $k\geq3$. If $G$ is affine, then according to Lemma~\ref{lemaffine}, either $k=3$ and $n=1$, or $k=4$ and $n$ is a power of $2$. The former leads to $G=G_{3,3}=\Sy_3$, which satisfies the conclusion of the theorem. For the latter, since $G$ is $2$-transitive, we conclude from~\eqref{eq:wr} that $n$ is not a power of $4$, and so $n$ is an odd power of $2$, again satisfying the conclusion of the theorem. Thus we may assume that $G$ is almost simple for the rest of the proof.

First assume that $k\geq4$. Take a transposition $\tau\in\Sym([k])$. By Lemma~\ref{lemfpr}, we have
%Note from Lemma~\ref{lemfpr} that for every transposition $\tau\in\Sy_{k}$,
\[
\fpr(\rho_\tau)=\frac{k-2}{k}\geq\frac{1}{2}.
\]
Then since $G$ is $2$-transitive, it follows from~\cite[Theorem~1]{GM1998} that either $G\geq\A_{kn}$, or
\begin{equation}\label{eq:6}
\fpr(\rho_\tau)=\frac{1}{2}+\frac{1}{2(2^r\pm1)}\ \text{ for some }\,r\geq3.
\end{equation}
The former satisfies the conclusion of the theorem. Now suppose that~\eqref{eq:6} holds. It follows that
\[
\fpr(\rho_\tau)\leq\frac{1}{2}+\frac{1}{2(2^3-1)}=\frac{4}{7}<\frac{3}{5}.
\]
This together with $\fpr(\rho_\tau)=(k-2)/k$ implies that $k<5$. Thus $k=4$, which in turn yields $\fpr(\rho_\tau)=(k-2)/k=1/2$, contradicting~\eqref{eq:6}.
%This completes the proof for $k\geq4$.

In the following assume that $k=3$. For convenience in the coming discussion, we first calculate $G_{3,3n}$ for $n\leq92$ by computation in~\textsc{Magma}~\cite{Magma}. It turns out that, for these values of $n$, if $n$ is not a power of $3$ then $G_{3,3n}$ contains $\A_{3n}$. Note by~\eqref{eq:wr} that if $n$ is a power of $3$ then $G_{3,3n}$ is not $2$-transitive. Thus in the remainder of the proof we assume $n>92$.

Suppose for a contradiction that $G$ does not contain $\A_{3n}$. Since $n>92$, it follows from the list of almost simple $2$-transitive groups (see~\cite[Table~7.4]{Cameron1999}) that $\Soc(G)$ is a simple group of Lie type, say, over $\mathbb{F}_q$. In the following we divide the proof into four cases according to $q>4$, $q=4$, $q=3$ or $q=2$. Take a transposition $\tau\in\Sym([3])$. We have  $\fpr(\rho_\tau)=1/3$ by Lemma~\ref{lemfpr}.

\textsf{Case~1}: $q>4$. In this case, $\fpr(\rho_\tau)=1/3>4/(3q)$. Then since $G$ is a $2$-transitive group on $3n>276$ points, it follows from~\cite[Theorem~1]{LS1991} that $\Soc(G)=\PSL(2,q)$ and $\fpr(\rho_\tau)$ is either $2/(q+1)$ or $(q_0+1)/(q+1)$, where $q_0=q^{1/r}$ is a prime power for some integer $r\geq2$. This together with $\fpr(\rho_\tau)=1/3$ implies that $1/3=2/(q+1)$ or $1/3\leq(\sqrt{q}+1)/(q+1)$. However, this leads to $q\leq9$, and hence $3n=q+1\leq10$, a contradiction.

\textsf{Case~2}: $q=4$. In this case, we see from~\cite[Table~7.4]{Cameron1999} that $G$ is a subgroup of either $\PGaU(3,4)$ or $\PGaL(d,4)$ with $d\geq2$, which together with $n>92$ implies that $G\leq\PGaL(d,4)$ with $d\geq3$. Then according to~\cite[Proposition~3.1]{GK2000}, the fixed point ratio of a non-identity element in $G$ is less than
\[
\min\left\{\frac{1}{2},\ \frac{1}{4}+\frac{1}{4^{d-1}}\right\}=\frac{1}{4}+\frac{1}{4^{d-1}}\leq\frac{1}{4}+\frac{1}{4^2}<\frac{1}{3},
\]
contradicting $\fpr(\rho_\tau)=1/3$.

\textsf{Case~3}: $q=3$. Recall that $G$ is a $2$-transitive group on $3n>276$ points. Then we see from~\cite[Table~7.4]{Cameron1999} that
$G$ is a subgroup of $\PGL(d,3)$ with $d\geq6$. It follows from Lemma~\ref{lempgl} that
\[
\frac{3^s+3^t-2}{3^d-1}=\fpr(\rho_\tau)=\frac{1}{3}
\]
for some nonnegative integers $s$ and $t$. This yields
\[
3(3^s+3^t-2)=3^d-1,
\]
which is not possible.

\textsf{Case~4}: $q=2$. In this case, we see from the list of almost simple $2$-transitive groups that either $G=\PSL(d,2)$ with $d\geq3$, or $G$ is the group $\Sp(2m,2)$ for some $m\geq3$ with point stabiliser $\mathrm{O}^\pm(2m,2)$.

First assume $G=\PSL(d,2)$ with $d\geq3$. Then $G$ can be viewed as $\GL(d,2)$ acting on the set of nonzero vectors. In this way, $\Fix(\rho_\tau)=\{v\in\mathbb{F}_2^d\mid v^x=v\}\setminus\{0\}$, and so $|\Fix(\rho_\tau)|=2^r-1$ for some nonnegative integer $r\leq d$.
It then follows from $\fpr(\rho_\tau)=1/3$ that
\[
\frac{1}{3}=\fpr(\rho_\tau)=\frac{2^r-1}{2^d-1}.
\]
This yields
\begin{equation}\label{eq:case4}
3\cdot2^r=2^d+2.
\end{equation}
Since the right hand side of~\eqref{eq:case4} is congruent to $2$ modulo $4$, we deduce $3\cdot2^r\equiv2\pmod{4}$ and thus $r=1$. However, this leads to $6=2^d+2$, contradicting $d\geq3$.

Now assume $G=\Sp(2m,2)$ for some $m\geq3$ with point stabiliser $\mathrm{O}^\varepsilon(2m,2)$, where $\varepsilon\in\{+,-\}$.
Note that $\rho_\tau$ is an involution with nonempty fixed point set. Let $H$ be a point stabiliser of $G$ containing $\rho_\tau$.
According to~\cite[(8.5)]{AS1976}, two involutions in $H$ are conjugate in $G$ if and only if they are conjugate in $H$. Hence $(\rho_\tau)^H=(\rho_\tau)^G\cap H$. Then by~\cite[Lemma~1.2(iii)]{Burness2018} we have
\[
\fpr(\rho_\tau)=\frac{|(\rho_\tau)^G\cap H|}{|(\rho_\tau)^G|}=\frac{|(\rho_\tau)^{H}|}{|(\rho_\tau)^{G}|}
=\frac{|H|\cdot|\mathbf{C}_{G}(\rho_\tau)|}{|G|\cdot|\mathbf{C}_{H}(\rho_\tau)|}
=\frac{|\mathrm{O}^\varepsilon(2m,2)|}{|\Sp(2m,2)|}\cdot\frac{|\mathbf{C}_{G}(\rho_\tau)|}{|\mathbf{C}_{H}(\rho_\tau)|}.
\]
This in conjunction with Lemma~\ref{lemspo} implies that
\[
\fpr(\rho_\tau)=\frac{|\mathrm{O}^\varepsilon(2m,2)|}{|\Sp(2m,2)|}\cdot\frac{|\Sp(2m-2r,2)|}{|\rmO^{\varepsilon}(2m-2r,2)|}
\cdot\frac{|\mathbf{O}_2(\mathbf{C}_{G}(\rho_\tau)|}{|\mathbf{O}_2(\mathbf{C}_{H}(\rho_\tau))|}
\]
for some positive integer $r\leq m$. According to whether $r=m$ or $r<m$, we deduce that
\[
\fpr(\rho_\tau)=\frac{1}{2^{m-1}(2^m+\varepsilon1)}\cdot\frac{|\mathbf{O}_2(\mathbf{C}_{G}(\rho_\tau)|}{|\mathbf{O}_2(\mathbf{C}_{H}(\rho_\tau))|}\ \text{ or }\ \frac{2^{m-r-1}(2^{m-r}+\varepsilon1)}{2^{m-1}(2^m+\varepsilon1)}
\cdot\frac{|\mathbf{O}_2(\mathbf{C}_{G}(\rho_\tau)|}{|\mathbf{O}_2(\mathbf{C}_{H}(\rho_\tau))|}.
\]
Since $\fpr(\rho_\tau)=1/3$ and both $|\mathbf{O}_2(\mathbf{C}_{G}(\rho_\tau)|$ and $|\mathbf{O}_2(\mathbf{C}_{H}(\rho_\tau))|$ are powers of $2$, it follows that
\[
\frac{1}{3}=\frac{1}{2^m+\varepsilon1}\ \text{ or }\ \frac{2^{m-r}+\varepsilon1}{2^m+\varepsilon1}.
\]
The former is not possible as $m\geq3$. For the latter, we obtain
\[
3\cdot2^{m-r}+\varepsilon2=2^m\equiv0\pmod{4}
\]
and thus $m-r=1$, which in turn leads to $m=3$ and $\varepsilon=+$. However, this implies that $3n=|\Sp(6,2)|/|\rmO^+(6,2)|=36$, contradicting $n>92$.
\end{proof}

\section{$2$-transitivity}\label{sec4}

We will prove Theorem~\ref{thm2tran} in this section.
Throughout this section, let $n=k^st$ where $s$ and $t$ are integers satisfying $s\geq 0$, $t>1$ and $k\nmid t$. For a nonnegative integer $m$ and a positive integer $\ell$, we use $[m]_{\ell}^{0}$ and $[m]_{\ell}^{1}$ to denote the remainder and quotient of $m$ divided by $\ell$, that is,
\[
m=\ell[m]_{\ell}^{1}+[m]_{\ell}^{0}
\]
with $0\leq[m]_{\ell}^{0}\leq\ell-1$.
For every $x\in[kn]$ (note that $0\le x<k^{s+1}t$), we write $[x]_{t}^{1}$ in base $k$ as follows: $[x]_{t}^{1}=k^sx_s+\cdots+kx_1+x_0$ where $x_i\in[k]$ for every $i\in[s+1]$. Therefore, $x$ can be uniquely written as
\begin{equation*}
x=(k^sx_s+\cdots+kx_1+x_0)t+[x]_{t}^{0}.
\end{equation*}
For convenience, we identify $x$ with $(x_s,\ldots,x_1,x_0;X)$ where $X=[x]_{t}^{0}$, and sometimes we mix the two notations when doing addition.
For example,
\begin{equation*}
(x_s,\ldots,x_3,0,1,1;t-1)+k^2t+2=(x_s,\ldots,x_3,1,1,2;1).
\end{equation*}
Recall $(i+jn)^{\sigma}=ki+j$ for all $i\in[n]$ and $j\in[k]$. One can obtain inductively that
\begin{align}
  (x_s,\ldots,x_1,x_0;X)^{\sigma^{i}}  &= \left(\sum_{j=i}^s k^{j}x_{j-i}\right)t+k^iX+\sum_{j=0}^{i-1}k^{i-1-j}x_{s-j}\nonumber\\
  &=(x_{s-i},\ldots,x_1,x_0,0,\ldots,0;0)+k^iX+\sum_{j=0}^{i-1}k^{i-1-j}x_{s-j}\label{spower}
\end{align}
for all $i\in [s+2]$ (when $i=s+1$, the tuple $(x_{s-i},\ldots,x_1,x_0,0,\ldots,0;0)$ in equation~\eqref{spower} is to be understood as $0$). In particular,
\begin{equation}\label{eq:sigmai}
     (x_s,\ldots,x_1,x_0;X)^{\sigma}=(x_{s-1},\ldots,x_0,[kX+x_s]_{t}^{1};[kX+x_s]_{t}^{0}),
\end{equation}
and thus
\begin{equation}\label{inverse}
(x_s,\ldots,x_1,x_0;X)^{\sigma^{-1}}=([x_0t+X]_{k}^{0},x_{s},\ldots,x_1;[x_0t+X]_{k}^{1}).
\end{equation}
Recalling from~\eqref{Eqn:rho} that $(i+jn)^{\rho_{\tau}}=i+j^{\tau}n$,
we have
\begin{equation}\label{ta}
(x_s,\ldots,x_1,x_0;X)^{\rho_{\tau}}=(k^{s-1}x_{s-1}+\cdots+kx_1+x_0)t+X+x_s^\tau n=(x_{s}^{\tau},x_{s-1},\ldots,x_1,x_0;X)
\end{equation}
for every $(x_s,\ldots,x_1,x_0;X)\in[kn]$. By~\eqref{spower} and~\eqref{ta}, it is clear that
\begin{equation}
\label{tau}
(x_s,\ldots,x_1,x_0;X)^{\sigma^{i}\rho_{\tau}\sigma^{-i}}
=(x_s,\ldots,x_{s-i+1},x_{s-i}^{\tau},x_{s-i-1},
\ldots,x_1,x_0;X).
\end{equation}

Consider the subgroup $H:=\langle\sigma,\rho_{\tau}\mid \tau \in \mathrm{Sym}([k-1])\rangle$ of $G_{k,kn}$,
which is contained in the stabiliser of $kn-1$. Our strategy for proving Theorem \ref{thm2tran} is to prove that $H$ is transitive on $[kn-1]$.
We use $(i,j)\in \Sym([k])$ with $i\neq j$ to denote the transposition swapping $i$ and $j$. For each $x\in[k]$, let $(0,x)$ denote the permutation of $[k]$ sending $x$ to $0$ and $0$ to $x$ while fixing $[k]\setminus \{0,x\}$ pointwise. In particular, $(0,x)$ coincides with the above notation for a transposition if $x\neq 0$ and is the identity permutation if $x=0$. This somewhat cumbersome notation avoids discussing whether $x=0$ in the following.

Let $x=(x_s,x_{s-1},\ldots,x_0;X)\in[kn]$.
Write $\alpha_i=\sigma^{i}\rho_{(0,1)}\sigma^{-i}$ for every $i\in[s+1]$.
By~\eqref{tau},
\begin{equation*}
x^{\alpha_i}
=(x_s,\ldots,x_{s-i+1},x_{s-i}^{(0,1)},x_{s-i-1},\ldots,x_0;X).
\end{equation*}
Set $\beta_{\tau}=\sigma^{-1}\rho_{\tau}\sigma\in H$ for $\tau\in \mathrm{Sym}([k-1])$. Using~\eqref{eq:sigmai}--\eqref{ta}, it is straightforward to check that
\[
x^{\beta_{\tau}}=x+([x_0t+X]_k^0)^{\tau}-[x_0t+X]_k^0.
\]
We will use the above two formulas for $\alpha_i$ and $\beta_{\tau}$ repeatedly without any reference.
Let
\[
  T(x)=|\{i\in[s+1]\mid x_i=k-1\}|.
\]

\begin{lemma}\label{lm:T0}
If $T(x)=0$, then $x\in0^{H}$.
\end{lemma}

\begin{proof}
It follows from $T(x)=0$ that $x_i\neq k-1$ for $i\in[s+1]$. This combined with~\eqref{tau} shows
\[x^{\prod_{i=0}^s\sigma^{s-i}\rho_{(0,x_i)}\sigma^{-(s-i)}}=(x_s^{(0,x_s)},\ldots,x_0^{(0,x_0)};X)=(0,\ldots,0;X)=X\in[t]\]
and $\prod_{i=0}^s\sigma^{s-i}\rho_{(0,x_i)}\sigma^{-(s-i)}\in H$. So it suffices to prove that $[t]\subseteq 0^H$. We achieve this by showing that $x^H$ contains an integer less than $x$ for each $x\in[t]\setminus\{0\}$.

Let $x\in [t]\setminus\{0\}$.
If $[x]_{k}^{0}=0$, then $x^{\sigma^{-1}} = x/k<x$.
If $[x]_{k}^{0}\neq 0$ and $[x]_{k}^{0}\neq k-1$, then $x^{\beta_{\tau}}=x-[x]_k^0<x$, where $\tau=(0,[x]_k^0)$.
If $[x]_{k}^{0}=k-1$ and $[t]_{k}^{0}\neq 1$, then $[t+x]_{k}^{0}\notin\{0,k-1\}$, and so there exists $\tau\in \Sym([k-1])$ such that $([t+x]_{k}^{0})^{\tau}=[t+x]_{k}^{0}-1$, which leads to
\[
 x^{\alpha_{s}\beta_{\tau}\alpha_{s}}=(t+x)^{\beta_{\tau}\alpha_{s}}=(t+x-1)^{\alpha_{s}}= x-1<x.
\]
If $[x]_{k}^{0}=k-1$  and $[t]_{k}^{0}=1$, then $[t+x]_{k}^{0}=0$  and hence
\[
 x^{\alpha_{s}\beta_{(0,1)}\alpha_{s}\sigma^{-1}}= (t+x)^{\beta_{(0,1)}\alpha_{s}\sigma^{-1}}=(t+x+1)^{\alpha_{s}\sigma^{-1}}
=(x+1)^{\sigma^{-1}}=(x+1)/k<x.
\]
Therefore, $[t]\subseteq0^{H}$, as desired.
\end{proof}

\begin{lemma}
\label{xy}
Let $x=(x_{s},x_{s-1},\ldots,x_1,x_0;X)\in [kn-1]$. If $1\leq T(x)\leq s$, then there exists $y=(y_{s},y_{s-1},\ldots,y_1,y_0;Y)\in x^{H}$ such that either $T(y)=0$, or $y_0=0$, $y_1=k-1$ and $T(x)\geq T(y)$.
\end{lemma}

\begin{proof}
Let $\ell$ be the smallest integer such that $x_\ell\ne k-1$. Write
\begin{equation*}
x^{\sigma^{-\ell}}=z=(z_{s},z_{s-1},\ldots,z_1,z_0;Z).
\end{equation*}
Applying~\eqref{inverse} repeatedly, we derive $z_{s-\ell}=x_s$, $z_{s-\ell-1}=x_{s-1},\ldots,z_{1}=x_{\ell+1}$, $z_0=x_{\ell}$. Since $x_0=\cdots=x_{\ell-1}=k-1$, it follows that $T(x)\geq T(z)$. If $T(z)=0$, then we confirm the lemma by taking $y=z$. In what follows, assume $T(z)>0$.

Since $z_0=x_{\ell}\ne k-1$ and $k\geq 3$,  there exists $\tau\in\Sym([k-1])$ such that $|z_0^{\tau}-z_0|=1$.
Set
\begin{equation*}
\mu_0=
  \begin{cases}
    \sigma^{-1} & \text{if}~[z_0t+Z]_{k}^{0}\ne k-1 \\
    \sigma^{s}\rho_{\tau}\sigma^{-s-1} & \text{if}~[z_0t+Z]_{k}^{0}=k-1.
  \end{cases}
\end{equation*}
Then by~\eqref{inverse} and \eqref{tau}, we obtain
\begin{equation*}
z^{\mu_0}=
  \begin{cases}
    \left([z_0t+Z]_k^0,z_s,\ldots,z_1;[z_0t+Z]_k^1\right) & \text{if}~[z_0t+Z]_{k}^{0}\ne k-1 \\
    \left([z_0^{\tau} t+Z]_k^0,z_s,\ldots,z_1;[z_0^{\tau}t+Z]_k^1\right) & \text{if}~[z_0t+Z]_{k}^{0}=k-1.
  \end{cases}
\end{equation*}
If both $[z_0t+Z]_{k}^{0}$ and $[z_0^{\tau}t+Z]_{k}^{0}$ are equal to $k-1$, then it follows from   $|z_0^{\tau}-z_0|=1$ that $k$ divides $t$, a contradiction. Thus $[z_0^{\tau}t+Z]_{k}^{0}\ne k-1$ if  $[z_0t+Z]_{k}^{0}= k-1$. Consequently, $T(z^{\mu_0})=T(z)$.

Let $j$ be the smallest integer such that $z_{j+1}=k-1$. Since, in particular, none of $z_1,\ldots,z_{j-1}$ is equal to $k-1$, along the same lines as the above paragraph we can take $\mu_1,\ldots,\mu_{j-1}\in H$ such that $z^{\mu_0\mu_1\cdots\mu_{j-1}}=(w_{s},w_{s-1},\ldots,w_1,w_0;W)$ with $w_0=z_j$, $w_1=z_{j+1}$ and
\begin{equation*}
T(z^{\mu_0\mu_1\cdots\mu_{j-1}})=\cdots=T(z^{\mu_0\mu_1})=T(z^{\mu_0})=T(z).
\end{equation*}
Let $w=z^{\mu_0\mu_1\cdots\mu_{j-1}}$ and $y=w^{\sigma^{s}\rho_{(0,w_0)}\sigma^{-s}}$. Since $w_0=z_j\neq k-1$ and $w_1=z_{j+1}=k-1$, it follows from~\eqref{tau} that
\[
y=(w_{s},w_{s-1},\ldots,w_1,w_0;W)^{\sigma^{s}\rho_{(0,w_0)}\sigma^{-s}}=(w_s,\ldots,w_2,k-1,0;W)
\]
and $T(y)=T(w)$. This together with $\sigma^{-\ell}\mu_0\mu_1\cdots\mu_{j-1}\sigma^{s}\rho_{(0,w_0)}\sigma^{-s}\in H$ and $T(x)\geq T(z)=T(w)$ completes the proof.
\end{proof}

\begin{lemma}\label{lm:Ts}
  If $T(x)=s+1$, then $x^H$ contains an integer less than $x$.
\end{lemma}

\begin{proof}
Since $x_i=k-1$ for every $i\in[s+1]$, we have
\begin{equation*}
x=(k-1)(k^s+\cdots+k+1)t+X=(k^{s+1}-1)t+X=k^{s+1}t-(t-X).
\end{equation*}
Observe that~\eqref{spower} implies
\[
 x^{\sigma^{s+1}} =k^{s+1}X+(k-1)\sum_{i=0}^{s} k^i =k^{s+1}(X+1)-1.
\]
Since $x=k^{s+1}t-(t-X)=kn-(t-X)<kn-1$, it follows that $X<t-1$ and so
\[
x-x^{\sigma^{s+1}}=k^{s+1}t-(t-X)-k^{s+1}(X+1)+1=
(k^{s+1}-1)(t-X-1)>0.
\]
Thus $x>x^{\sigma^{s+1}}$.
\end{proof}

Now we are ready to prove Theorem \ref{thm2tran}.

\begin{proof}[Proof of Theorem \ref{thm2tran}]
Recall our notation that $n=k^st$ with $s\geq0$, $t>1$ and $k\nmid t$, and $H$ is the subgroup of $G_{k,kn}$ generated by $\sigma$ and $\rho_{\tau}$ for all $\tau\in \Sym([k-1])$. Then $H$ is contained in the stabiliser of $kn-1$ in $G_{k,kn}$. Since $G_{k,kn}$ is transitive, it is $2$-transitive if $H$ is transitive on $[kn-1]$. Thus it suffices to prove that $[kn-1]\subseteq 0^H$.  Let
\begin{equation*}
x=(x_{s},x_{s-1},\ldots,x_1,x_0;X)\in [kn-1],
\end{equation*}
where $X\in[t]$ and $x_i\in[k]$ for $i\in[s+1]$.
Recall that $T(x)=|\{i\in[s+1]\mid x_i=k-1\}|$. We show $x\in 0^{H}$ for all $x\in[kn-1]$ by induction on $T(x)$. The base case $T(x)=0$ has been confirmed by Lemma~\ref{lm:T0}. Now let $T(x)\geq 1$ and suppose that $y\in 0^{H}$ for all $y\in[kn-1]$ with $T(y)<T(x)$.
We will complete the proof by constructing $y\in x^H$ such that $T(y)<T(x)$.

If $T(x)=s+1$, then since $x$ is finite, we derive by using Lemma~\ref{lm:Ts} repeatedly that there exists $y\in x^H$ with $T(y)<T(x)$.
In the following we assume that $1\le T(x)<s+1$.  By Lemma~\ref{xy}, we can further assume $x=(x_s,\ldots,x_2,k-1,0;X)$.
The proof proceeds in two cases.

\textsf{Case~1}: $[t]_k^0\neq k-1$.

Let $z=(x_s,\ldots,x_2,k-1,0;Z)$, where $Z=X+1-[X+1]_k^0\equiv 0\pmod{k}$.
We first show in the next paragraph that $z\in x^H$.

If $[X]_k^0\neq k-1$, then letting $\tau=(0, [X]_k^0)$, we have
\[
  x^{\beta_{\tau}}=(x_s,\ldots,x_2,k-1,0;X-[X]_k^0)=(x_s,\ldots,x_2,k-1,0;X+1-[X+1]_k^0)=z.
\]
Now assume $[X]_k^0= k-1$. Then $[t+X]_k^0\neq k-1$ and $[t+X]_k^0-[t]_k^0=-1$. Letting $\tau=([t]_k^0, [t+X]_k^0)$, we have
  \begin{align*}
    x^{\alpha_s\beta_{\tau}\alpha_s}&=(x_s,\ldots,x_2,k-1,0;X)^{\alpha_s\beta_{\tau}\alpha_s}\\
    &=(x_s,\ldots,x_2,k-1,1;X)^{\beta_{\tau}\alpha_s}\\
    &=(x_s,\ldots,x_2,k-1,1;X+[t]_k^0-[t+X]_k^0)^{\alpha_s}\\
    &=(x_s,\ldots,x_2,k-1,0;X+1)\\
    &=(x_s,\ldots,x_2,k-1,0;X+1-[X+1]_k^0)\\
    &=(x_s,\ldots,x_2,k-1,0;Z).
  \end{align*}
Therefore, $z=(x_s,\ldots,x_2,k-1,0;Z)\in x^H$.

In view of~\eqref{inverse}, we obtain that
\[
  z^{\sigma^{-2}}=(0,x_s,\ldots,x_2,k-1;Z/k)^{\sigma^{-1}}=([(k-1)t+Z/k]_k^0,0,x_s,\ldots,x_2;[(k-1)t+Z/k]_k^1)
\]
and that, with $W:=(k-1)t+(Z+t-[t]_k^0)/k$,
\begin{align*}
  z^{\alpha_s\sigma^{-2}}&=(x_s,\ldots,x_2,k-1,1;Z)^{\sigma^{-2}}\\
  &=\left([t]_k^0,x_s,\ldots,x_2,k-1;\frac{Z+t-[t]_k^0}{k}\right)^{\sigma^{-1}}=\left([W]_k^0,[t]_k^0,x_s,\ldots,x_2;[W]_k^1\right).
\end{align*}
If  $[(k-1)t+Z/k]_k^0\neq k-1$ or $[W]_k^0\neq k-1$, then taking $y=z^{\sigma^{-2}}$ or $y=z^{\alpha_s\sigma^{-2}}$, respectively, we have $y\in z^H=x^H$ and $T(y)=T(x)-1<T(x)$. This completes the proof for the case $[(k-1)t+Z/k]_k^0\neq k-1$ or $[W]_k^0\neq k-1$.

Next assume $[(k-1)t+Z/k]_k^0= k-1=[W]_k^0$, or equivalently, $t-1\equiv Z/k\pmod{k}$ and $(t-[t]_k^0)/k \equiv 0\pmod{k}$.
If $Z/k=1$, then $[t]_k^0=2$, which together with the assumption of Case 1 implies that $k>3$, and hence $\beta_{(0,2)}\in H$. Thus, taking
\begin{align*}
  y&=z^{\alpha_s\beta_{(0,2)}\alpha_s\sigma^{-2}}\\
  &=(x_s,\ldots,x_2,k-1,1;k)^{\beta_{(0,2)}\alpha_s\sigma^{-2}}\\
  &=(x_s,\ldots,x_2,k-1,1;k-2)^{\alpha_s\sigma^{-2}}\\
  &=(x_s,\ldots,x_2,k-1,0;k-2)^{\sigma^{-2}}\\
  &=(k-2,x_s,\ldots,x_2,k-1;0)^{\sigma^{-1}}\\
  &=(k-2,k-2,x_s,\ldots,x_2; [(k-1)t]_k^1)
\end{align*}
we have $y\in z^H=x^H$ and $T(y)=T(x)-1<T(x)$, as desired.
Similarly, if $Z/k\geq 2$, then as $[t+Z]_k^0=[t]_k^0\notin\{0,k-1\}$, taking $\tau=([t]_k^0, [t]_k^0-1)$, $\mu=(k-2 , k-3)$ and
\begin{align*}
 y&=z^{\alpha_s\beta_{\tau}\alpha_s\sigma^{-1}\beta_{\mu}\sigma\alpha_s\sigma^{-2}}\\
  &=(x_s,\ldots,x_2,k-1,0;Z-1)^{\sigma^{-1}\beta_{\mu}\sigma\alpha_s\sigma^{-2}}\\
  &=(k-1,x_s,\ldots,x_2,k-1;Z/k-1)^{\beta_{\mu}\sigma\alpha_s\sigma^{-2}}\\
  &=(k-1,x_s,\ldots,x_2,k-1;Z/k-2)^{\sigma\alpha_s\sigma^{-2}}\\
  &=(x_s,\ldots,x_2,k-1,1;Z-k-1)^{\sigma^{-2}}\\
  &=\left([t-1]_k^0, x_s,\ldots,x_2,k-1;\frac{Z+t-[t]_k^0}{k}-1\right)^{\sigma^{-1}}\\
  &=([W-1]_k^0,[t-1]_k^0, x_s,\ldots,x_2;[W-1]_k^1)\\
  &=(k-2,[t-1]_k^0, x_s,\ldots,x_2;[W-1]_k^1),
\end{align*}
we have $y\in z^H=x^H$ and $T(y)=T(x)-1<T(x)$, as desired. If $Z=0$, then $[t]_k^0=1$, which implies that
\begin{equation*}
  z^{\alpha_s\beta_{(0,1)}}=(x_s,\ldots,x_2,k-1,1;0)^{\beta_{(0,1)}}=(x_s,\ldots,x_2,k-1,0;t-1),
\end{equation*}
and then the previous two sentences show that there exists $y\in(z^{\alpha_s\beta_{(0,1)}})^H=x^H$ with $T(y)<T(x)$.

\textsf{Case~2}: $[t]_{k}^{0}=k-1$.

Recall that $x=(x_s,\ldots,x_2,k-1,0;X)$. Let
\[
  u=(0,x_s,\ldots,x_2,k-1;U)\quad \text{and} \quad v=(0,x_s,\ldots,x_2,k-1;V),
\]
where $U=(X-[X]_k^0)/k$ and $V=(t+X+1-[X]_k^0)/k$. We first show that $u,v\in x^H$.

If $[X]_{k}^{0}=k-1$, then $[t+X]_k^0=k-2$, and it follows that
\begin{align*}
  x^{\alpha_s\beta_{(k-2,k-3)}\alpha_s\beta_{(0,k-2)}\sigma^{-1}}
   &= (x_s,\ldots,x_2,k-1,1;X)^{\beta_{(k-2,k-3)}\alpha_s\beta_{(0,k-2)}\sigma^{-1}}\\
   &= (x_s,\ldots,x_2,k-1,1;X-1)^{\alpha_s\beta_{(0,k-2)}\sigma^{-1}}\\
   &= (x_s,\ldots,x_2,k-1,0;X-1)^{\beta_{(0,k-2)}\sigma^{-1}}\\
   &= (x_s,\ldots,x_2,k-1,0;X-k+1)^{\sigma^{-1}}\\
   &= (0,x_s,\ldots,x_2,k-1;U)\\
   &= u.
\end{align*}
 If $[X]_{k}^{0}\neq k-1$, then letting $\tau=(0,[X]_k^0)$, we have
\[
  x^{\beta_{\tau}\sigma^{-1}}= (x_s,\ldots,x_2,k-1,0;X-[X]_k^0)^{\sigma^{-1}}=u.
\]
Hence it always holds that $u\in x^H$. If $[t+X]_{k}^{0}=k-1$, then $[X]_k^0=0$ and thus
\begin{align*}
  x^{\beta_{(0,1)}\alpha_s\sigma^{-1}}&=(x_s,\ldots,x_2,k-1,0;X+1)^{\alpha_s\sigma^{-1}} \\
  &= (x_s,\ldots,x_2,k-1,1;X+1)^{\sigma^{-1}}=(0,x_s,\ldots,x_2,k-1;V)=v.
\end{align*}
 If $[t+X]_{k}^{0}\neq k-1$, then as $[t+X]_k^0=[X]_k^0-1$, we obtain by taking $\tau=(0,[t+X]_k^0)$ that
\begin{align*}
  x^{\alpha_s\beta_{\tau}\sigma^{-1}}
  &= (x_s,\ldots,x_2,k-1,1;X)^{\beta_{\tau}\sigma^{-1}}\\
  &= (x_s,\ldots,x_2,k-1,1;X-[t+X]_k^0)^{\sigma^{-1}}=(0,x_s,\ldots,x_2,k-1;V)=v.
\end{align*}
Therefore, $v\in x^H$ always holds as well.

Now we have proved $u,v \in x^H$. If $[(k-1)t+U]_{k}^{0}\neq k-1$, then since
\[
u^{\sigma^{-1}}=([(k-1)t+U]_{k}^{0},0,x_s,\ldots,x_2;[(k-1)t+U]_{k}^1),
\]
it follows that $u^{\sigma^{-1}}\in x^H$ with $T(u^{\sigma^{-1}})=T(x)-1<T(x)$.
Similarly, if $[(k-1)t+V]_{k}^{0}\neq k-1$, then $v^{\sigma^{-1}}\in x^H$ with $T(v^{\sigma^{-1}})=T(x)-1<T(x)$.
This completes the proof for the case $[(k-1)t+U]_{k}^{0}\ne k-1$ or $[(k-1)t+V]_{k}^{0}\neq k-1$.

Next assume $[(k-1)t+U]_{k}^{0}=k-1=[(k-1)t+V]_{k}^{0}$. Since $X-[X]_k^0\leq X\leq t-1$,
\begin{align*}
  u^{\sigma\alpha_s\sigma^{-1}}&= (x_s,\ldots,x_2,k-1,0;X-[X]_k^0)^{\alpha_s\sigma^{-1}}\\
  &= (x_s,\ldots,x_2,k-1,1;X-[X]_k^0)^{\sigma^{-1}}
  = (k-1,x_s,\ldots,x_2,k-1;V-1).
\end{align*}
Moreover, we deduce from $[(k-1)t+U]_{k}^{0}=k-1$ that $U\geq [U]_k^0=k-2\geq 1$, which implies $V\geq2$ and $0\leq X-[X]_k^0-k\leq t-1$. This combined with $[(k-1)t+U]_{k}^{0}=k-1=[(k-1)t+V]_{k}^{0}$ yields that
\begin{align*}
  u^{\sigma\alpha_s\sigma^{-1}\beta_{(k-2,k-3)}\sigma\alpha_s\sigma^{-2}}&
  =(k-1,x_s,\ldots,x_2,k-1;V-1)^{\beta_{(k-2,k-3)}\sigma\alpha_s\sigma^{-2}}\\
  &=(k-1,x_s,\ldots,x_2,k-1;V-2)^{\sigma\alpha_s\sigma^{-2}}\\
  &=(x_s,\ldots,x_2,k-1,1;X-[X]_k^0-k)^{\alpha_s\sigma^{-2}}\\
  &=(x_s,\ldots,x_2,k-1,0;X-[X]_k^0-k)^{\sigma^{-2}}\\
  &=(0,x_s,\ldots,x_2,k-1;U-1)^{\sigma^{-1}}\\
  &=(k-2,0,x_s,\ldots,x_2;[(k-1)t+U-1]_k^1).
\end{align*}
As a consequence, with $y:=u^{\sigma\alpha_s\sigma^{-1}\beta_{(k-2,k-3)}\sigma\alpha_s\sigma^{-2}}\in u^H=x^H$, we finally obtain that $T(y)=T(x)-1<T(x)$.
\end{proof}

\section{Open problems on generalised shuffle groups}\label{sec5}

Shuffle groups on $kn$ cards can be considered in a more general way by restricting the permutations on the set of $k$ piles to a subgroup of $\Sym([k])$. Precisely, if $P\leq \Sym([k])$ is a group of permutations on the set of $k$ piles, then we define the \emph{generalised shuffle group} on $kn$ cards with respect to $P$ by
\[
  \Sh(P,n):=\langle \rho_\tau\sigma \mid \tau\in P\rangle=\langle \sigma, \rho_\tau \mid \tau\in P\rangle,
\]
where $\sigma$ is the standard shuffle and $\rho_\tau$ is the permutation on $kn$ cards induced by the permutation $\tau$ on the $k$ piles.
In particular, $\Sh(\Sym([k]),n)$ is exactly the group $G_{k,kn}$ studied in this paper. Generalised shuffle groups are introduced and systematically studied by Amarra, Morgan and Praeger in~\cite{AMP2021}. Among several open problems, a conjecture~\cite[Conjecture~1.10]{AMP2021} made by them is that if $k\geq3$, $n$ is not a power of $k$ and $(k,n)\neq(4,2^f)$ for any positive integer $f$, then $\Sh(C_k,n)$ contains $\A_{kn}$, where $C_k$ is generated by the $k$-cycle $(0,1,\ldots,k-1)\in \Sym([k])$.

Note that $\Sh(C_k,n)=\langle \sigma, \rho_{(0,1,\ldots,k-1)}\sigma\rangle$. Hence the above mentioned conjecture asserts that, somewhat surprisingly, two shuffles $\sigma$ and $\rho_{(0,1,\ldots,k-1)}\sigma$ are enough to generate $\A_{kn}$ or $S_{kn}$. This suggests that a ``best possible" improvement to Theorem~\ref{classification} would be the determination of $\Sh(C_k,n)$. It is shown in~\cite[Theorem~1.4(1)]{AMP2021} that if $kn=k^m$, then
\begin{equation}\label{eq:k_power}
  \Sh(P,n)=P\wr C_m
\end{equation}
for any $P\leq \Sym([k])$. If $k=4$ and $kn=2^m$ with $m$ odd, then similarly to the proof of~\cite[Theorem~2.6]{CHMW2005} we derive that
\begin{equation}\label{eq:AGL}
  \Sh(C_k,n)=\AGL(m,2).
\end{equation}
According to~\cite[Lemma~2]{MM1987}, the standard shuffle $\sigma$ is an even permutation if and only if
\begin{equation}\label{eq:parity_sigma}
  \frac{k(k-1)}{2}\cdot\frac{n(n-1)}{2}\equiv 0\pmod{2}.
\end{equation}
Observing that $\rho_{(0,1,\ldots,k-1)}$ is a product of $n$ cycles of length $k$, we obtain that $\rho_{(0,1,\ldots,k-1)}$ is even if and only if $(k-1)n$ is even. Hence $\Sh(C_k,n)\leq \A_{kn}$ if and only if
\[
  \frac{k(k-1)}{2}\cdot\frac{n(n-1)}{2}\equiv (k-1)n\equiv 0\pmod{2}.
\]
% either $n\equiv2\pmod4$ and $k\equiv0$ or $1\pmod4$ with $(k,n)\neq(4,2)$, or $n\equiv0\pmod4$ and $n$ is not a power of $k$.
This together with~\eqref{eq:k_power} and~\eqref{eq:AGL} indicates that~\cite[Conjecture~1.10]{AMP2021} is essentially the following conjectural classification of $\Sh(C_k,n)$ for all $k\geq 3$ and $n\geq 1$.

% \begin{conjecture}\label{conj:C_k}
%   If $k\geq 3$ and $C_k=\langle(0,1,\ldots,k-1)\rangle\leq \Sym([k])$, then the following hold:
%       \begin{enumerate}[\rm(a)]
%             \item If $kn=k^m$, then $\Sh(C_k,n)$ is the primitive wreath product $C_k\wr C_m$.
%             \item If $k=4$ and $kn=2^m$ with $m$ odd, then $\Sh(C_k,n)=\AGL(m,2)$.
%             \item If $(k,n)\neq (4,2^f)$ for all $f$ odd, $n$ is not a power of $k$, either $k\equiv0$ or $1\pmod4$ or $n\equiv 0$ or $1\pmod4$, and $(k-1)n\equiv 0\pmod{2}$, then $\Sh(C_k,n)=A_{kn}$.
%             \item In all other cases, $\Sh(C_k,n)=S_{kn}$.
%       \end{enumerate}
% \end{conjecture}

\begin{conjecture}\label{conj:C_k}
  If $k\geq 3$ and $C_k=\langle(0,1,\ldots,k-1)\rangle\leq \Sym([k])$, then the following hold:
      \begin{enumerate}[\rm(a)]
            \item If $kn=k^m$, then $\Sh(C_k,n)$ is the primitive wreath product $C_k\wr C_m$.
            \item If $k=4$ and $kn=2^m$ with $m$ odd,  then $\Sh(C_k,n)$ is the affine group $\AGL(m,2)$.
            \item If $n$ is not a power of $k$ and either $k(k-1)n(n-1)/4$ or $(k-1)n$ is odd, then $\Sh(C_k,n)=S_{kn}$.
            \item In all other cases, $\Sh(C_k,n)=A_{kn}$.
      \end{enumerate}
\end{conjecture}

A choice of $P$ to make $\Sh(P,n)$ close to $G_{k,kn}$ is $P=\A_k$. For the case $kn=k^m$, it is already known (see~\eqref{eq:k_power}) that $\Sh(\A_k,n)=\A_k\wr C_m$. Moreover, since $\rho_\tau$ is even for each $\tau\in \A_k$, the parity of $\sigma$ implies that $\Sh(\A_k,n)\leq A_{kn}$ if and only if~\eqref{eq:parity_sigma} holds. Therefore, we pose the following conjectural classification of $\Sh(\A_k,n)$ for all $k\geq 3$ and $n\geq 1$.
% Using a similar argument to the proof of~\cite[Theorem~2.6]{CHMW2005}, it also can be shown that if $k=4$ and $kn=2^m$ with $m$ odd,  then
% \begin{equation}
%   \Sh(\A_k,n)=\AGL(m,2).
% \end{equation}
% \begin{conjecture}\label{conj:A_k}
%   If $k\geq 3$, then the following hold:
%     \begin{enumerate}[\rm(a)]
%           \item If $kn=k^m$, then $\Sh(\A_k,n)$ is the primitive wreath product $\A_k\wr C_m$.
%           \item If $k=4$ and $kn=2^m$ with $m$ odd,  then $\Sh(\A_k,n)$ is the affine group $\AGL(m,2)$.
%           \item If $(k,n)\neq (4,2^f)$ for all $f$ odd, $n$ is not a power of $k$ and either $k\equiv0$ or $1\pmod4$ or $n\equiv 0$ or $1\pmod4$, then $\Sh(\A_k,n)=A_{kn}$.
%           \item In all other cases, $\Sh(\A_k,n)=S_{kn}$.
%     \end{enumerate}
% \end{conjecture}
\begin{conjecture}\label{conj:A_k}
  If $k\geq 3$, then the following hold:
    \begin{enumerate}[\rm(a)]
          \item If $kn=k^m$, then $\Sh(\A_k,n)$ is the primitive wreath product $\A_k\wr C_m$.
          \item If $k=4$ and $kn=2^m$ with $m$ odd,  then $\Sh(\A_k,n)$ is the affine group $\AGL(m,2)$.
          \item If $n$ is not a power of $k$ and $k(k-1)n(n-1)/4$ is odd, then $\Sh(\A_k,n)=S_{kn}$.
          \item In all other cases, $\Sh(\A_k,n)=A_{kn}$.
    \end{enumerate}
\end{conjecture}
Proving this conjecture should be easier than proving Conjecture~\ref{conj:C_k}. For one reason, if $k$ is odd, then $\Sh(C_k,n)\leq \Sh(\A_k,n)$, and so the conclusion of Conjecture~\ref{conj:A_k} is weaker in this case. For another reason, Conjecture~\ref{conj:A_k} is closer to our Theorem~\ref{classification} in the sense that the size of $P$ is only reduced by half from $P=S_k$ to $P=A_k$. Thus, some ideas in the proof of Theorem~\ref{classification} also apply to Conjecture~\ref{conj:A_k}. For example, $\rho_{(0,1,2)}\in \Sh(\A_k,n)$ has fixed point ratio $(k-3)/k$, which is at least $1/2$ when $k\geq 6$. In this way, a parallel result to Theorem~\ref{thmreduct} might still be established by the approach of this paper with an ad hoc treatment for $k\in\{3,4,5\}$.
However, we anticipate more work to be done to  prove the $2$-transitivity of $\Sh(\A_k,n)$.

As a contrast to $P=C_k$ or $A_k$, the choice $P=\langle \mathrm{Rev}(k)\rangle$ from~\cite[Page~6]{MM1987}, where $\mathrm{Rev}(k)$ is the permutation on $[k]$ sending $i$ to $k-1-i$, will make $\Sh(P,n)$ never equal to $G_{k,kn}$. In fact, denoting
\[
  R_{k,kn}=\Sh(\langle \mathrm{Rev}(k)\rangle,n)
\]
and $B_i=\{i,kn-1-i\}$ (can be a singleton if $i=kn-i-1$) for $i\in\{0,1,\ldots, \lfloor (kn-1)/2 \rfloor\}$, we can verify directly  that $R_{k,kn}$ preserves the set $\{B_0,B_1,\ldots,B_{\lfloor (kn-1)/2 \rfloor}\}$. If $kn$ is even, then $\{B_0,B_1,\ldots,B_{\lfloor (kn-1)/2 \rfloor}\}$ is a block system of $R_{k,kn}$, and so $R_{k,kn}\leq C_2\wr S_{kn/2}$ is imprimitive. If $kn$ is odd, then $R_{k,kn}$ fixes the $((kn-1)/2)$-th card and preserves the partition $\{B_0,B_1,\ldots,B_{\lfloor (kn-1)/2 \rfloor-1}\}$ of the rest $kn-1$ cards, which implies that $R_{k,kn}$ is intransitive with $R_{k,kn}\leq C_2\wr S_{(kn-1)/2}$. We have the following conjecture based on computation results.

\begin{conjecture}\label{conj:rev}
  Let $k\geq 2$, $n\geq 2$ and $R_{k,kn}=\Sh(\langle \mathrm{Rev}(k)\rangle,n)$. Suppose that $kn$ is even and $(k,n)\neq (\ell^e,\ell^f)$ for any positive integers $\ell$, $e$ and $f$. Then the following hold:
\begin{enumerate}[\rm(a)]
  \item If $(k,n)=(2,6)$ or $(6,2)$, then $R_{k,kn}=C_2^6\rtimes \PGL(2,5)$.
  \item If $(k,n)=(3,4)$, then $R_{k,kn}=A_5$.
  \item If $(k,n)=(4,3)$, then $R_{k,kn}=C_2\times A_5$.
  \item If $kn=24$, then $R_{k,kn}=C_2^{11}\rtimes M_{12}$.
  \item If $kn\neq 12$, $k\equiv 2$ or $3\pmod{4}$ and $n\equiv 2\pmod{4}$, then  $R_{k,kn}=C_2\wr S_{kn/2}$.
  \item If $k\equiv 2\pmod{4}$ and $n\equiv 1\pmod{4}$, then $R_{k,kn}=C_2^{ kn/2}\rtimes A_{ kn/2}$.
  \item If $k\equiv 2\pmod{4}$ and $n\equiv 3\pmod{4}$, then $R_{k,kn}=C_2^{ (kn-2)/2}\rtimes S_{kn/2}$.
  \item Otherwise, $R_{k,kn}=C_2^{ (kn-2)/2}\rtimes A_{ kn/2}$.
\end{enumerate}
\end{conjecture}

\begin{remark}
Statements~(a)--(d) have been verified by computation in~\textsc{Magma}~\cite{Magma}, and we include them in Conjecture~\ref{conj:rev} for completeness. In fact, statement~(d) is already mentioned in~\cite{MM1987}. The reason why we assume $kn$ even and $(k,n)\neq (\ell^e,\ell^f)$ is that we have not yet identified the patterns of $R_{k,kn}$ if $kn$ is odd or $(k,n)= (\ell^e,\ell^f)$ for some positive integers $\ell$, $e$ and $f$. However, we do have some interesting observations in special cases. For example, for $(k,n)=(\ell^e,\ell^f)$ with $\gcd(e,f)=1$, it seems that
\begin{equation}\label{eq:R_{k,kn}}
  R_{k,kn}=\begin{cases}
    C_2^{e+f-1}\rtimes C_{e+f}&\ \text{ if }\  e\equiv f+1\equiv 0\pmod{2}\\
    C_2\wr C_{e+f}&\ \text{ otherwise}.
  \end{cases}
\end{equation}
  This would be a generalisation of~\cite[Theorem~1.4(1)]{AMP2021}, as the latter can be obtained from~\eqref{eq:R_{k,kn}} by taking $e=1$.
  % Another observation is that if $kn$ is even, either $n\not\equiv 2\pmod{4}$ or $k\not\equiv 3\pmod{4}$, and either $k\neq 6$ or $n$ is even, then $R_{k,kn}\leq G_{2,kn}$.
\end{remark}

Finally, we would like to pose the following more challenging question.
\begin{question}\label{ques}
Given $k\geq 3$ and $n\geq 1$ such that $n$ is not a power of $k$ and $(k,n)\neq(4,2^f)$ for any odd integer $f$, for what $\theta\in\Sym([k])$ does $\langle\sigma,\rho_{\theta}\sigma\rangle=\Sh(\langle\theta\rangle,n)$ contain $\A_{kn}$?
\end{question}
Note that a complete answer to Question~\ref{ques} would in particular solve Conjectures~\ref{conj:C_k} and~\ref{conj:rev}. Another interesting consequence would be the proportion
\[
  \frac{\{\theta\in \Sym([k])\mid \Sh(\langle\theta\rangle,n) \text{ contains } \A_{kn}\}}{k!}
\]
of valid permutations $\theta$ in $\Sym([k])$ for a pair $(k,n)$, especially when $k$ and $n$ are large. Our computation results suggest that this proportion is at least (for most cases much larger than) $1/6$.

% \begin{Backmatter}
\vspace{15pt}
\noindent\textbf{Acknowledgement.}
The second author was supported by the Natural Science Foundation of Chongqing (CSTB2022NSCQ-MSX1054). The third author was supported by the Melbourne Research Scholarship provided by The University of Melbourne. The fourth author was supported by the China Scholarship Council (202106040068). The work was done during a visit of the fourth author to The University of Melbourne. The fourth author would like to thank The University of Melbourne for its hospitality and Beijing Normal University for consistent support. The authors wish to express their sincere gratitude to the anonymous referee for careful reading and invaluable suggestions to improve this paper.

% \end{Backmatter}

\end{document}